\RequirePackage{fix-cm}
\RequirePackage{amsmath}
\documentclass[smallextended, numbook]{svjour3} 

\smartqed{}
\usepackage{mathptmx}
\journalname{Noname}

\usepackage{amsmath}
\usepackage{amssymb}
\usepackage{hyperref}

\hypersetup{
    unicode=true,
    pdftitle={Exact Coupling of Random Walks on Polish Groups},
    pdfauthor={James T. Murphy III},
    colorlinks,
    linktocpage,
    citecolor=blue,
    linkcolor=blue,
    urlcolor=blue,
    filecolor=blue
}

\renewcommand{\leq}{\leqslant}
\renewcommand{\geq}{\geqslant}

\newcommand{\R}{\mathbb{R}}
\newcommand{\Q}{\mathbb{Q}}
\newcommand{\F}{\mathcal{F}}
\newcommand{\A}{\mathcal{A}}
\newcommand{\B}{\mathcal{B}}
\newcommand{\N}{\mathbb{N}}

\renewcommand{\P}{\mathbf{P}}
\newcommand{\E}{\mathbf{E}}
\newcommand{\Z}{\mathbb{Z}}
\newcommand{\T}{\mathcal{T}}
\newcommand{\zero}{0}
\newcommand{\TV}[1]{\left\Vert{#1}\right\Vert_{\mathrm{TV}}}
\newcommand{\textdefn}[1]{\emph{#1}}
\DeclareMathOperator{\supp}{supp}
\DeclareMathOperator{\RW}{RW}
\DeclareMathOperator{\len}{len}

\begin{document}

\title{Exact Coupling of Random Walks on Polish Groups}

\author{James~T.~Murphy~III}
\institute{James~T.~Murphy~III \at 
The University of Texas at Austin \\
\email{james@intfxdx.com}
}
\date{Received: 21 Jun 2017 / Accepted: 30 Aug 2018}
\journalname{Journal of Theoretical Probability}

\maketitle

\begin{abstract}
    Exact coupling of random walks is studied.
    Conditions for admitting a successful exact coupling are given that are necessary and in the
    Abelian case also sufficient.
    In the Abelian case,
    it is shown that a random walk $S$ with step-length distribution $\mu$
    started at $0$ admits a successful exact coupling
    with a version $S^x$ started at $x$ if and only if there is $n\geq 1$ with $\mu^{n} \wedge \mu^{n}(x+\cdot) \neq \zero$. 
	Moreover, when a successful exact coupling exists,
	the total variation distance between $S_n$ and $S^x_n$ is
	determined to be $O(n^{-1/2})$ if $x$ has infinite order, or
	$O(\rho^n)$ for some $\rho \in (0,1)$ if $x$ has finite order.
    In particular, this paper solves a problem posed by H.~Thorisson on successful exact coupling of random
    walks on $\R$.
    It is also noted that the set of such $x$ for which a successful exact
    coupling can be constructed is a Borel measurable group.
    Lastly, the weaker notion of possible exact coupling and its relationship
    to successful exact coupling are studied.
\end{abstract}

\section*{Acknowledgments}
This work was supported by a grant of the Simons Foundation (\#197982 to The University of Texas at Austin).
The author also thanks the anonymous referees for their helpful suggestions.

\section{Introduction}
Let $G$ be a Polish group with identity $e$.
Recall that a Polish group is a group equipped with a topology under which
multiplication and inversion are continuous operations, and such that the topology is
separable and completely metrizable.
The most poignant examples to keep in mind throughout are $\R^d$ and $\Z^d$.
If $G$ is Abelian, additive notation is used instead and the identity is denoted $0$.
Fix, for the remainder of the document, a Borel probability measure $\mu$ on
$G$.
For each $x \in G$,
let $\RW(x,\mu)$ be the law of a (right) random walk on $G$ started at $x$ and with 
\textdefn{step-length distribution} $\mu$. 
That is, $\RW(x,\mu)$ is the law on the product space $G^\N$ of a
process $S^x = \{S^x_n\}_{n=0}^\infty$ such that
\begin{equation}
    S^x_n = xX_1 X_2\dotsm X_n, \qquad 0 \leq n < \infty,
\end{equation}
where the \textdefn{step-lengths} $\{X_i\}_{i=1}^\infty$ are i.i.d.\ random elements in $G$ with distribution $\mu$.
Such a process $S^x \sim \RW(x,\mu)$ is called an \textdefn{$(x,\mu)$-random walk.}

One may be interested in the long-term effects of the choice of the initial
location $x \in G$ of a random walk.
After a long time, can one distinguish an $(x,\mu)$-random walk from a
$(y,\mu)$-random walk in the sense of total variation?
That is, for $(x,\mu)$- and $(y,\mu)$-random walks $S^x$
and $S^y$, one would like to know whether
\begin{align}\label{eq:TVconvergence}
    \TV{\P(S^x_n \in \cdot) - \P(S^y_n \in \cdot)} \to 0,\qquad n \to \infty,
\end{align}
where $\TV{\nu} = \sup_{B}\nu(B) - \inf_{B}\nu(B)$ denotes the
total variation of a finite signed measure $\nu$.
Note that for probability measures $\nu_1$ and $\nu_2$, one also has
\begin{equation}
	\TV{\nu_1-\nu_2} = 2 \sup_{B}|\nu_1(B) - \nu_2(B)|.
\end{equation}
An equivalent formulation of \eqref{eq:TVconvergence}
may be expressed in terms of successful exact
couplings.

An \textdefn{exact coupling} of $\RW(x,\mu)$ and $\RW(y,\mu)$
is a triple $(S^x, S^y, T)$ defined on a 
probability space $(\Omega,\F,\P)$ in such a way that $S^x\sim \RW(x,\mu)$, $S^y
\sim \RW(y,\mu)$, and $T$ is a random time,
called a \textdefn{coupling time},
such that
\begin{equation}
    S^x_n = S^y_n,\qquad n \geq T.
\end{equation}
If $T$ is a.s.\ finite, the exact coupling is called \textdefn{successful}.
If $\tilde S^x\sim \RW(x,\mu)$ and $\tilde S^y\sim \RW(y,\mu)$ are defined on
possibly different spaces, one also calls
$(S^x, S^y, T)$ an exact coupling of $\tilde S^x$ and $\tilde S^y$.
The condition \eqref{eq:TVconvergence} is
equivalent to the statement that $\RW(x,\mu)$ and $\RW(y,\mu)$ admit a successful
exact coupling.
See Theorem 9.4 in Section 9.5 of \cite{thorisson2000coupling} for the equivalence of these
statements.
In essence, successful exact coupling may be achieved if and only if the
initial condition is uniformly forgotten as
time progresses.
Moreover, the tail probabilities of a coupling time control the speed at
which the total variation distance between the two random walks decays.
Indeed, if $(S^x, S^y, T)$ is an exact coupling as above, one has for any Borel set $B
\subseteq G$,
\begin{align*}
	\left|\P(S^x_n \in B) - \P(S^y_n \in B)\right|
	= \left| \E[1_{S^x_n \in B} - 1_{S^y_n \in B}] \right|
	\leq \P(T > n),\qquad n \geq 0.
\end{align*}
Multiplying by $2$ and taking the supremum over all Borel $B$, one finds
\begin{align}\label{eq:couplinginequality}
	\TV{\P(S^x_n \in \cdot) - \P(S^y_n \in \cdot)} \leq 2\P(T > n), \qquad n
	\geq 0.
\end{align}

This paper investigates under what conditions successful exact couplings may be
constructed.
When successful exact coupling can be achieved, the
constructed coupling
time is analyzed to give bounds on the rate at which total variation distance
decays.
Note that if $S^x$ is an $(x,\mu)$-random walk and $y \in G$, then $y S^x$ is
a $(yx, \mu)$-random walk. 
Hence $\RW(x,\mu)$ and $\RW(y,\mu)$ admit a successful exact coupling if and only
if
$\RW(e,\mu)$ and $\RW(y^{-1}x,\mu)$ admit a successful exact coupling.
It therefore suffices to study only the case when one of the initial locations
is the identity.
That is, for what initial positions $x$ do $\RW(e,\mu)$ and $\RW(x,\mu)$ admit
a successful exact coupling?
\begin{definition}
	Define the \textdefn{successful exact coupling set} $G_s$ to be the subset of
	all $x \in G$ such that there exists a successful exact coupling of
	$\RW(e,\mu)$ and $\RW(x,\mu)$.
\end{definition}
The primary question is then to determine what is the set $G_s$.
This question was posed in \cite{thorisson2011open} for $G:=\R$, and in that case
the following two special cases were known as early as 1965, cf.\
\cite{stam1966shifting,herrmann1965glattungseigenschaften}, though 
references using more modern notation are cited here.
In the following, recall that when $G$ admits a (left-invariant) Haar measure
$\lambda$ (e.g.\ Lebesgue measure on $\R^d$ or a counting measure on $\Z^d$),
then 
$\mu$ is called \textdefn{spread out} if for some $n \geq 1$ one has $\mu^n \geq \int_\cdot f\,d\lambda$
for some Borel $f\geq 0$ not $\lambda$-a.e.\ zero, where for a Borel measure $\nu$
on $G$, $\nu^n$ denotes the $n$-fold convolution of $\nu$ with itself.

\begin{theorem}\label{thm:spreadoutcase}\cite{thorisson2000coupling}
	Let $G:=\R$.
	Then $G_s = G$ if and only if the step-length distribution
    $\mu$ is spread out.
\end{theorem}

\begin{theorem}\label{thm:atomiccase}\cite{arnaldsson2010coupling}
    Let $G:=\R$ and
    suppose $\mu$ is purely atomic with $A$ denoting the set of atoms of $\mu$.
	Then $G_s$ is the subgroup generated by $A-A= \{a-a': a,a' \in A\}$.
\end{theorem}

Given two Borel measures $\nu_1$ and $\nu_2$ on $G$, denote $\nu_1 \wedge \nu_2$ to be the largest measure smaller
than $\nu_1$ and $\nu_2$.
The zero measure is denoted $\zero$.
For $x \in G$, also define the shift $\theta_x\nu$ by $\theta_x\nu(B) := \nu(x^{-1}B)$
for each Borel $B \subseteq G$.
The interpretation of $\theta_x\nu$ is $\nu$ with all mass shifted (left-multiplied) by $x$,
and $\theta_x\nu$ satisfies 
$\int_G f(y)\,\theta_x\nu(dy) = \int_G f(xy)\,\nu(dy)$
for all Borel $f:G\to\R_{\geq 0}$.

The resolution to Thorisson's problem and generalizations of the previous
theorems may now be stated.
The proof is postponed and broken into several separate more general theorems appearing across multiple sections.

\begin{theorem}\label{thm:abelianproblemresolutionteaser}
    Suppose $G$ is Abelian. 
    Then the following hold:
    \begin{enumerate}
		\item[(a)]{} 
			$G_s = \{x \in G: \exists n \geq 1, \mu^n \wedge
			\theta_x^{-1}\mu^n \neq 0\}$.
		\item[(b)]{}
			For $x \in G_s$ and $n_0 \geq 1$ such that $\mu^{n_0} \wedge
			\theta_x^{-1}\mu^{n_0}\neq \zero$, there is $C=C(\mu,x,n_0) > 0$ such that for $S \sim
			\RW(0,\mu)$ and $S^x \sim \RW(x,\mu)$ under $\P$, one has
			\begin{itemize}
				\item if $x$ has infinite order,
                    \begin{align}
						\TV{\P(S_n \in \cdot) - \P(S^x_n \in \cdot)} \leq
					\frac{C}{\sqrt{n}},\qquad n \geq 1,
					\end{align}
				\item if $x$ has finite order, there is
					$\rho=\rho(\mu,x,n_0) \in (0,1)$
					such that 
					\begin{align}
						\TV{\P(S_n \in \cdot) - \P(S^x_n \in \cdot)} \leq C
					\rho^n \qquad n \geq 1.
					\end{align}
			\end{itemize}
        \item[(c)]{} Suppose $G$ is locally compact with Haar measure
		$\lambda$.
			If $G_s = G$, then $\mu$ is spread out.
            If $G$ is connected, the converse holds as well.
        \item[(d)]{} Suppose $\mu$ is purely atomic with $A$ denoting the set
			of atoms of $\mu$.
			Then $G_s$ is the subgroup generated by $A-A$.
		\item[(e)]{} $G_s$ is a Borel measurable subgroup of $G$.
    \end{enumerate}
\end{theorem}

\section{Outline of the Paper}
 
Section~\ref{sec:mainthm} builds to the main theorem,
Theorem~\ref{thm:mainthm},
which generalizes Theorem~\ref{thm:abelianproblemresolutionteaser} (a) and (b).
It is more technical but also applies in some non-Abelian cases.
The reader familiar with the proof of the spread out case~\cite{thorisson2000coupling}
or the purely atomic case \cite{arnaldsson2010coupling} on $\R$ may recognize the proof of
Proposition~\ref{prop:existence}, which shows that if $\mu^{n}$ dominates the
sum of a measure $\nu$ and some shift $\theta_x^{-1}\nu$, then successful exact coupling can be achieved.
Similarly, the proof of the main theorem of this document
follows the spirit of the purely atomic case on $\R$.

With the main theorem proved, Section~\ref{sec:abeliancase} covers parts (a)-(d)
of 
Theorem~\ref{thm:abelianproblemresolutionteaser} as simple corollaries.
That is, it resolves the Abelian case and gives an even simpler description of
$G_s$
in the spread out and purely atomic cases.
Note that Corollary~\ref{cor:connectedspreadoutcase} is more general than 
claimed in Theorem~\ref{thm:abelianproblemresolutionteaser} (c),
as one direction applies in the non-Abelian case without extra restrictions.

Section~\ref{sec:structureofcouplingset} then investigates the structure of $G_s$.
In particular, it is shown to be Borel measurable.
This, together with the fact that $G_s$ is a group,
shows part (e) of Theorem~\ref{thm:abelianproblemresolutionteaser}.

Finally, in Section~\ref{sec:possibleexactcoupling},
the weaker notion of possible exact coupling is studied.
It is noted that the necessary and sufficient conditions derived for successful
exact coupling in the Abelian case are
coincidental.
It is shown that the conditions derived for admitting a
successful exact coupling in the Abelian case are,
in the general case, equivalent to the ostensibly weaker notion of admitting a possible exact coupling, 
but that in the Abelian case admitting a possible exact coupling and
admitting a successful exact coupling are equivalent.
The paper ends with an example on a (non-Abelian) free group for which possible
exact coupling can be done but successful exact coupling
cannot.

\section{The Main Theorem}\label{sec:mainthm}

This section culminates
in the main theorem of the paper, Theorem~\ref{thm:mainthm},
which gives necessary and sometimes sufficient conditions for successful exact coupling to occur,
even in the non-Abelian case.
The transfer and splitting theorems that
appear in \cite{thorisson2000coupling} are used.
Less general versions are stated that are sufficient for the current setting.
The need of a Polish space in the following is also the primary reason $G$ is assumed
to be Polish.

\begin{theorem}[Transfer Theorem]\label{thm:transfer}\cite{thorisson2000coupling}
    Suppose $(\Omega,\F,\P)$ is a probability space and $Y_1$ is a random element in
    $(E_1,\mathcal{E}_1)$.
    Further suppose that there is a pair $(Y_1',Y_2')$ on some probability space $(\Omega',\F',\P')$
    with $Y_2'$ a random element in a Polish space $(E_2,\mathcal{E}_2)$, and $Y_1$ is a version of $Y_1'$.
    Then $Y_2'$ can be transferred to $(\Omega,\F,\P)$, i.e.\ $(\Omega,\F,\P)$ can be extended to accommodate
    a random element $Y_2$ which is conditionally independent of the original space given $Y_1$,
    and with $(Y_1,Y_2)$ having the same distribution as $(Y_1',Y_2')$.
    This transfer procedure can be repeated countably many times.
\end{theorem}

\begin{theorem}[Splitting Theorem]\label{thm:splitting}\cite{thorisson2000coupling}
    Suppose $(\Omega, \F,\P)$ is a probability space and $Y$ is a random element in $(E,\mathcal{E})$.
    Let $\{\nu_i\}_{i=0}^\infty$ be subprobability measures on $(E,\mathcal{E})$ and suppose
    $\P(Y \in \cdot) \geq \sum_i \nu_i$.
    Then $(\Omega,\F,\P)$ can be extended to accommodate a nonnegative integer-valued random variable $K$,
    called a splitting variable, such that $\P(Y \in \cdot, K=i) = \nu_i$.
    Moreover, $K$ is conditionally independent of the original space given $Y$.
    This splitting operation can be repeated countably many times.
\end{theorem}

Theorems~\ref{thm:transfer} and~\ref{thm:splitting} are useful for constructing random variables with specific dependencies
on a single probability space.
A simple example application is the following.
If one can construct up to a countable number of successful exact couplings,
then in fact they can be made to occur on the same probability space.

\begin{proposition}\label{prop:singlecouplingspace}
	If $N \in \N \cup \{\infty\}$ and $\{x_i\}_{i=1}^N \subseteq G_s$,
	then there exists
	a single probability space $(\Omega,\F,\P)$ housing $S\sim\RW(e,\mu)$ and
	$S^{x_i}\sim \RW(x_i,\mu)$ for each $i$ such that
    for every $i$ there is an a.s.\ finite random time $T_i$ with $S^{x_i}_n = S_n$ for all $n \geq T_i$.
	That is, $(S, S^{x_i}, T_i)$ is a successful exact coupling for all $i$.
\end{proposition}

\begin{proof}
	By assumption, there is a successful exact coupling of
	$\RW(e,\mu)$ and $\RW(x_1,\mu)$ on some $(\Omega,\F,\P)$.
	Since $x_i \in G_s$ for each $i$, the extension procedure given by Theorem~\ref{thm:transfer} 
    can be repeated countably many times, once for each $S^{x_i}$,
    to give a single extension of $(\Omega,\F,\P)$ on which $S$ and $S^{x_i}$ couple for every $i$.
    \qed{}
\end{proof}

This gives the first structural result about the successful exact coupling set.

\begin{corollary}\label{cor:couplinggroup}
	$G_s$ is a group.
\end{corollary}

\begin{proof}
    Consider $x,y \in G_s$. By Proposition~\ref{prop:singlecouplingspace}, respectively define successful exact couplings
    $(S,S^x, T^x)$ and $(S, S^y, T^y)$ on a common probability space.
    Then for $n \geq \max\{T^x,T^y\}$ it holds that $S^y_n = S_n = S^x_n$.
    In particular, $(x^{-1}S^x, x^{-1}S^y, \max\{T^x, T^y\})$ is a successful exact coupling
    of $\RW(e,\mu)$ and $\RW(x^{-1}y,\mu)$.
    Thus $x^{-1}y \in G_s$.
    \qed{}
\end{proof}

The following works towards determining a specific scenario when a successful exact coupling
can be constructed and it is an extension of a result of {\"O}.~Arnaldsson in \cite{arnaldsson2010coupling}
with nearly identical proof.

\begin{proposition}\label{prop:existence}
    Fix $x \in G$ and suppose that $n\geq1$ is
    such that $\mu^{n} \geq \nu + \theta_x^{-1}\nu$ for a nonzero measure $\nu$.
	If $G$ is Abelian or, more generally, if there is $B$ with $\mu^{n}(B)=1$
	such that $x$ commutes with all of $B$, then $x \in G_s$.
	In this case, $\RW(e,\mu)$ and $\RW(x,\mu)$ admit a
    successful exact coupling with a coupling time
    $T$ for which $T/n$ has the same distribution as the hitting time of $e$ of a 
    lazy simple symmetric random walk on the cyclic group $\langle x \rangle$ started at $x$ with probability
    $1-2\nu(G)$ of not moving at each step.
    In particular, $\P(T = n) = \nu(G)$.
\end{proposition}

\begin{proof}
	Begin with an $(\Omega,\F,\P)$ housing $S\sim \RW(e,\mu)$ with step-lengths $\{X_i\}_{i=1}^\infty$.
	An extension of $(\Omega,\F,\P)$ and an $S^x\sim \RW(x,\mu)$ on that extension are constructed
    such that successful exact coupling occurs.
    Let 
    \begin{equation}
        L_i := X_{(i-1)n+1}\dotsm X_{in}
    \end{equation}
    for $i \geq 1$ so that $\{L_i\}_{i=1}^\infty$ is an i.i.d.\ family
    and $\P(L_i \in \cdot) \geq \nu + \theta_x^{-1}\nu$.
    By Theorem~\ref{thm:splitting},
    expand $(\Omega,\F,\P)$ to accommodate random variables
    $\{K_i\}_{i=1}^\infty$ taking values in $\{0,1,2\}$ such that
    $\{(L_i,K_i)\}_{i=1}^\infty$ is an i.i.d.\ sequence
    and 
    \begin{equation}\label{eq:Lkdefn}
        \P(L_i \in \cdot, K_i=1) = \nu, \qquad \P(L_i \in \cdot, K_i=2) = \theta_x^{-1}\nu.
    \end{equation}
    For $i \geq 1$ define
    \begin{equation}
        L_i ' :=
        \begin{cases}
            L_i, & K_i = 0,\\
            x^{-1}L_i, & K_i=1,\\
            xL_i, & K_i=2.
        \end{cases}
    \end{equation}
    It is elementary to check using~\eqref{eq:Lkdefn} that $L_i'$ has the same distribution as $L_i$.
    Let $R$ be the random walk started at $e$ with step-lengths $\{L_i\}_{i=1}^\infty$,
    and let $R'$ be the random walk started at $x$ with
    step-lengths $\{L_i'\}_{i=1}^\infty$.
    By construction, $L_i'L_i^{-1} \in \{e, x, x^{-1}\}$.
    By assumption, it is possible to choose $B$ with $\mu^{n}(B)=1$ such that $x$ commutes with all of $B$.
    Thus, a.s.\ every $L_i,L_i' \in B$ and so a.s.\ for every $i$,
    \begin{equation*}
        R_i'R_i^{-1} = xL_1'\dotsm L_i' L_i^{-1} \dotsm L_1^{-1}
        = x(L_1'L_1^{-1})\dotsm(L_i'L_i^{-1}) \in \langle x \rangle = \{x^m: m \in \Z\}.
    \end{equation*}
    Thus, $R'R^{-1}$ is in distribution the same as a lazy simple symmetric random walk started from $x$ with step lengths
    $\{L_i'L_i^{-1}\}_{i=1}^\infty$.
    The walk has probability $\nu(G)$ to increase the power of $x$, $\nu(G)$ to
	decrease it, and $1-2\nu(G)<1$ to stay put at each time step.
    Since nontrivial lazy simple symmetric random walks on cyclic groups are recurrent, 
    there is an a.s.\ finite random time $M$ with $R'_MR^{-1}_M = e$, i.e.\ the random walks $R'$ and $R$
    meet at time $M$.
    Theorem~\ref{thm:transfer} makes it possible to extend $(\Omega,\F,\P)$
    one final time to accommodate an i.i.d.\ sequence $\{X_i''\}_{i=1}^\infty$ with each $X_i''$ having distribution $\mu$
    and such that $L_i' = X_{(i-1)n}''\dotsm X_{in}''$ for $i \geq 1$.
    Define $T := Mn$ and let $S^x$ be the random walk started at $x$ with step-lengths
    \begin{equation}
        X_i' :=
        \begin{cases}
            X_i'', & i \leq T,\\
            X_i, & i > T.
        \end{cases}
    \end{equation}
    Then $S$ and $S^x$ witness the definition of successful exact coupling with coupling time $T$.
    If $K_1=1$, then $R$ and $R'$ meet in one time step, so $T=n$, showing $\P(T = n) = \P(K_1=1) = \nu(G)$.
    \qed{}
\end{proof}

In the previous proof, the problem is reduced to the case where
a difference process (or, in the non-Abelian case, something that resembles a difference process) is a random walk. 
Since a general random walk may be transient,
it is important to the proof that the difference process is made to be a random walk not on all of $G$, 
but rather
on the cyclic group generated by $x$, so that the analysis reduces to that of $\Z$ or $\Z/d\Z$.
This highlights the fact that the joint distribution of $S$ and $S^x$ required
to cause successful exact coupling
is very special, and could not, except in trivial cases, be achieved with $S$ and $S^x$ being independent
before the coupling time $T$.

The main theorem of the document follows.

\begin{theorem}[Main Theorem]\label{thm:mainthm}
    Fix $x \in G$.
	If $x \in G_s$,
	then there is $n\geq 1$ such that
    $\mu^{n} \wedge \theta_x^{-1}\mu^{n} \neq \zero$.
    Conversely, if $n_0 \geq 1$ is such that $\mu^{n_0}\wedge
	\theta_x^{-1}\mu^{n_0}
	\neq \zero$ and there exists
	$B$ with $\mu^{n_0}(B)=1$ such that $x$ commutes with all of $B$,
    then $x \in G_s$.
	In this case,
	there exists a successful exact coupling of $\RW(e,\mu)$ and $\RW(x,\mu)$
    with a coupling time $T$ satisfying $\P(T = n_0) > 0$.
    Moreover, there is $C=C(\mu,x,n_0) > 0$ such that for $S \sim
	\RW(0,\mu)$ and $S^x \sim \RW(x,\mu)$ under $\P$, one has
	\begin{itemize}
		\item if $x$ has infinite order,
			\begin{align}
				\TV{\P(S_n \in \cdot) - \P(S^x_n \in \cdot)} \leq
			\frac{C}{\sqrt{n}},\qquad n \geq 1,
			\end{align}
		\item if $x$ has finite order, there is $\rho=\rho(\mu,x,n_0) \in (0,1)$
			such that 
			\begin{align}
				\TV{\P(S_n \in \cdot) - \P(S^x_n \in \cdot)} \leq C
			\rho^n \qquad n \geq 1.
			\end{align}
	\end{itemize}
\end{theorem}

\begin{proof}
	Suppose that $S\sim\RW(e,\mu)$ and $S^x\sim\RW(x,\mu)$ witness the definition of successful exact coupling with coupling time $T$ and respective
    step-lengths $\{X_i\}_{i=1}^\infty$ and $\{X_i'\}_{i=1}^\infty$.
    Choose $n$ such that $\P(T = n)>0$, which is possible since $T$ is a.s.\
	finite.
	Then one has the following comparisons of measures,
	\begin{align*}
		\zero 
		&\neq \P(T=n, S_n = S^x_n \in \cdot)\\
		&\leq \P(S_n = S^x_n \in \cdot)\\
		&\leq \P(S_n \in \cdot) \wedge \P(S^x_n \in \cdot)\\
		&= \mu^n \wedge \theta_x \mu^n.
	\end{align*}
	Applying $\theta_x^{-1}$ to both sides of the previous inequality then
	gives $\zero \neq \theta_x^{-1}\mu^n \wedge \mu^n$.

    Conversely, suppose $n_0$ is such that $\xi:=\mu^{n_0} \wedge
	\theta_x^{-1}\mu^{n_0} \neq 0$
	and that there exists
	$B$ with $\mu^{n_0}(B)=1$ such that $x$ commutes with all of $B$.
	In case $x=e$, $\RW(e,\mu)$ and $\RW(x,\mu)$ clearly admit a successful exact coupling with a coupling time $T:= 0$, so assume $x \neq e$.
    Choose $y\in \supp{\xi}$.
    Since $y \neq xy$, it is possible to choose a neighborhood $U$ of $y$ small enough that $U \cap xU = \emptyset$.
    Consider 
    \begin{equation}
        \nu := \xi((x^{-1}\cdot ) \cap U) \neq \zero.
    \end{equation}
    Then
    \begin{equation*}
        \nu \leq \mu^{n_0}\left(x\left( (x^{-1}\cdot) \cap U\right) \right)
        = \mu^{n_0}(\cdot \cap xU)
    \end{equation*}
    and
    \begin{equation*}
        \theta_x^{-1}\nu = \xi(\cdot \cap U) \leq \mu^{n_0}(\cdot \cap U).
    \end{equation*}
    It follows that
    \begin{equation}
        \nu + \theta_x^{-1}\nu \leq \mu^{n_0}(\cdot \cap (U \cup xU)) \leq
		\mu^{n_0}.
    \end{equation}
    
	Proposition~\ref{prop:existence} then shows $\RW(e,\mu)$ and $\RW(x,\mu)$ admit a
    successful exact coupling
    with a coupling time $T$ satisfying $\P(T=n_0)= \nu(G) >0$
    and such that $\tau := T/n_0$ has the distribution of the hitting time to $e$
    of a symmetric lazy random walk on $\langle x \rangle$ with
    $1-2\nu(G)$
    chance of not moving at each step.
    
	Suppose that $x$ has finite order $d$.
    In this case, $\tau$  also has the distribution
    of the hitting time to $0$ of the symmetric lazy random walk on
    $\Z/d\Z$ that is started at $1$ and absorbed when it hits $0$.
	Call $P$ the transition kernel the absorbing walk.
    The absorbing walk converges geometrically quickly to its stationary
    distribution $\delta_0$, cf.~\cite{levin2017markov}.
    Choose $\rho \in (0,1)$ and $C > 0$ such that
	$\TV{P^n(1,\cdot) - \delta_0} \leq C \rho^n$ for all $n \geq 1$.
	By \eqref{eq:couplinginequality}
	it suffices to show that $\P(T > n)$ decays geometrically as $n \to
	\infty$.
	Indeed, for all $n \geq 1$,
    \begin{align*}
	    \P(T > n)
	    &= \P(\tau > n/n_0)\\
	    &= P^{\lfloor n/n_0 \rfloor+1}(1,\{0\}^c)\\
	    &\leq \TV{P^{\lfloor n/n_0 \rfloor+1}(1,\cdot)-\delta_0} +
	    \delta_0(\{0\}^c)\\
	    &\leq C \rho^{\lfloor n/n_0\rfloor+1} + 0\\
		&\leq \tilde C \left(\tilde \rho\right)^n
    \end{align*}
	for new constants $\tilde C > 0$ and $\tilde \rho \in (0,1)$,
    as desired.

    Next suppose that $x$ has infinite order, so $\langle x \rangle \simeq \Z$.
    The tail decay of symmetric lazy random walks on $\Z$ are known,
    see, for example, Corollary 2.28 in~\cite{levin2017markov}.
    In~\cite{levin2017markov}, lazy random walks are defined to have chance
    $1/2$ of staying still at each step, but allowing a $1-2\nu(G) \in (0,1)$
    chance of staying still at each step does not modify the
    result beyond giving a different leading constant in the decay
    rate.
    Hence, by Corollary 2.28 in~\cite{levin2017markov}, choose $C > 0$ 
    such that $\P(\tau > n) \leq \frac{C}{\sqrt{n}}$ for integers $n \geq 1$.
    Thus, for all $n \geq 1$,
    \begin{align*}
	    \P(T > n) 
	    &= \P(\tau > n/n_0)\\
	    &= \P(\tau > \lfloor n/n_0 \rfloor+1)\\
	    &\leq \frac{C}{\sqrt{\lfloor n/n_0 \rfloor +1}}\\
		&\leq \frac{\tilde C}{\sqrt{n}}
    \end{align*}
	for some new constant $\tilde C > 0$,
    as desired.
    \qed{}
\end{proof}

\section{The Abelian Case}\label{sec:abeliancase}

In this section, parts (a)-(d) of 
Theorem~\ref{thm:abelianproblemresolutionteaser} are derived as simple
corollaries of the main theorem.
Firstly, 
determining $G_s$
can be resolved entirely for Abelian $G$.
This is parts (a) and (b) of Theorem~\ref{thm:abelianproblemresolutionteaser}.

\begin{corollary}\label{cor:abeliancase}
    Suppose $G$ is Abelian.
	Then $G_s = \{x \in G: \exists n\geq 1, \mu^{n} \wedge \theta_x^{-1}\mu^{n}
	\neq \zero\}$.
    Moreover, for $x \in G_s$ and $n_0 \geq 1$ such that $\mu^{n_0} \wedge
			\theta_x^{-1}\mu^{n_0}\neq \zero$, there is $C=C(\mu,x,n_0) > 0$ such that for $S \sim
			\RW(0,\mu)$ and $S^x \sim \RW(x,\mu)$ under $\P$, one has
			\begin{itemize}
				\item if $x$ has infinite order,
					\begin{align*}
						\TV{\P(S_n \in \cdot) - \P(S^x_n \in \cdot)} \leq
					\frac{C}{\sqrt{n}},\qquad n \geq 1,
					\end{align*}
				\item if $x$ has finite order, there is
					$\rho=\rho(\mu,x,n_0) \in (0,1)$
					such that 
					\begin{align*}
						\TV{\P(S_n \in \cdot) - \P(S^x_n \in \cdot)} \leq C
					\rho^n \qquad n \geq 1.
					\end{align*}
            \end{itemize}
\end{corollary}

\begin{proof}
    Since $G$ is Abelian, the condition in Theorem~\ref{thm:mainthm} that there
	is $B$ with $\mu^{n}(B)=1$ such that $x$ commutes 
    with all of $B$ is automatic.
    \qed{}
\end{proof}

Next, a generalization of part (c) of Theorem~\ref{thm:abelianproblemresolutionteaser} is covered.
That is, for connected spaces step-lengths are spread out if and only if a successful exact coupling can
always be achieved.
The only if direction is essentially the same as in \cite{berbee1979random},
Theorem 5.3.2, and it also applies in the non-Abelian setting.

\begin{corollary}\label{cor:connectedspreadoutcase}
    Suppose $G$ is locally compact with Haar measure $\lambda$.
	If $G_s = G$, then
	then $\mu$ is spread out.
    If $G$ is connected and Abelian, the converse holds as well.
    More generally, if $G$ is Abelian but not necessarily connected, then $G_s$ is clopen.
\end{corollary}

\begin{proof}
	Suppose $\RW(e,\mu)$ and $\RW(x,\mu)$ admit a successful exact coupling for all $x \in G$.
    Then for all $x \in G$, $\TV{\mu^n - \theta_x^{-1}\mu^n} \to 0$ as $n \to \infty$.
    Consequently,
    \begin{equation}
        G = \bigcup_{n =1}^\infty\{x\in G :  \TV{\mu^n - \theta_x^{-1}\mu^n}\leq 1\}.
    \end{equation}
    The measurability of the sets 
    $B_n:= \{x \in G:  \TV{\mu^n - \theta_x^{-1}\mu^n} \leq 1\}$ for $n \geq 1$
	is taken for
	granted here. This fact is proved in the upcoming
	Corollary~\ref{cor:couplingmeasurability}.
    Choose $n$ large enough that $\lambda(B_n)>0$.
    Suppose for contradiction that a Borel set $N \subseteq G$ is such that $\mu^n(N) = 1$ but $\lambda(N)=0$. 
    Then $\lambda(N^{-1})=0$ as well, and 
    \begin{align*}
        0 
        &= \int_G\int_{G} 1_{sx\in B_n}  1_{x \in N^{-1}}\, \lambda(dx)\, \mu^n(ds)\\
        &= \int_G\int_{G} 1_{x \in B_n}1_{s^{-1}x \in N^{-1}}\, \lambda(dx)\, \mu^n(ds)\\
        &= \int_G\int_{B_n} 1_{s \in x N}\, \lambda(dx)\, \mu^n(ds)\\
        &= \int_{B_n} \theta_x^{-1}\mu^n(N)\, \lambda(dx)\\
        &\geq \int_{B_n} \mu^n(N) - |\theta_x^{-1}\mu^n(N) - \mu^n(N)|\,\lambda(dx)\\
	&\geq \int_{B_n} 1 - \frac{1}{2}\TV{\theta_x^{-1}\mu^n - \mu^n}\,\lambda(dx)\\
        &\geq \frac{1}{2} \lambda(B_n) \\
        &> 0
    \end{align*}
    which is a contradiction.
    It follows that $\mu^n$ must not be singular with respect to $\lambda$, and hence $\mu$
    is spread out.

    For the other direction, suppose $G$ is Abelian
    and that $\nu:=\mu^{n} \geq \int_{\cdot} f\,d\lambda$ as stated.
    By replacing $f$ with $\min\{f,b\}1_{K}$ for some $b>0$ and $K\subseteq G$ compact, 
    one may assume $f$ is bounded and compactly supported.
    Furthermore, it is claimed that by replacing $n$ with $2n$ one may assume $f>\epsilon$ on some nonempty
    open set for some $\epsilon > 0$.
    Indeed,
    \begin{equation*}
        \mu^{2n} 
        = \nu * \nu
        \geq\int_{\cdot} f*f\,d\lambda.
    \end{equation*}
    Since $f$ is bounded and compactly supported, the convolution $f*f$ is continuous, 
    and also $\left\Vert f * f\right\Vert_{L^1} = \left\Vert f\right\Vert_{L^1}^2 > 0$, so $f$
    is not constant $0$.
    Thus the assumption that $f>\epsilon>0$ on some nonempty open set $U$ and for some $\epsilon>0$ is justified.
    In particular, 
    choosing a symmetric neighborhood $V$ of the identity such that $(U-x) \cap U \neq \emptyset$ for each $x \in V$,
    it holds that 
    \begin{equation*}
        \nu \wedge \theta_x^{-1}\nu(G) \geq \int_G \min \{f(y), f(x+y)\}\,\lambda(dy)
        \geq \int_{(U-x) \cap U} \epsilon \lambda(dy) > 0
    \end{equation*}
    for every $x \in V$.
	It follows that 
	$G_s \supseteq V$.
    By Corollary~\ref{cor:couplinggroup},
    $G_s$ is a subgroup of $G$, and thus $G_s$ is either clopen or has empty interior.
    Since $G_s$ contains the nonempty open set $V$, $G_s$ must be clopen.
    If $G$ is connected then this implies $G_s = G$.
    \qed{}
\end{proof}

The connectedness assumption in Corollary~\ref{cor:connectedspreadoutcase}
plays a nontrivial role.
For example, consider when $G$ is a countable group.
Then any choice of $\mu$ is automatically purely atomic because $G$ is countable and spread out
because the Haar measure is a counting measure.
The following corollary shows that in that case the conclusion of
Corollary~\ref{cor:connectedspreadoutcase} does not hold.
This is also part (d) of Theorem~\ref{thm:abelianproblemresolutionteaser}.

\begin{corollary}\label{cor:abeliandiscretecase}
    Suppose $G$ is Abelian and $\mu$ is purely atomic with $A$ the set of atoms of $\mu$.
	Then $G_s$ is the
	subgroup generated by $A-A$.
\end{corollary}

\begin{proof}
    The atoms of $\mu^{n}$ are $nA := A + \dotsb + A$.
    Then since $\mu^{n}$ is atomic, $\mu^{n} \wedge \theta_x^{-1}\mu^{n} \neq \zero$ if and only if
    $nA \cap (nA-x) \neq \emptyset$ if and only if $x \in nA - nA = n(A-A)$.
    Finally, note that $\bigcup_{n=1}^\infty n(A-A)$ is exactly the subgroup
    generated by $A-A$ since $A-A$ is symmetric.
    Corollary~\ref{cor:abeliancase} then finishes the claim.
    \qed{}
\end{proof}

The section ends by showing that, in the Abelian case, any countable subgroup can be a successful exact coupling set,
and that the Haar measure is insufficient to measure the size of $G_s$.

\begin{corollary}\label{cor:Gsanycountableset}
	Suppose $G$ is Abelian and $H$ is a countable subgroup of $G$.
	Then there is a choice of $\mu$ for which $G_s= H$.
\end{corollary}

\begin{proof}
    Any purely atomic $\mu$ whose set of atoms is $H$ suffices.
    If $\mu$ is as mentioned, then since the subgroup generated by $H-H$ is $H$ itself,
	one finds that $G_s = H$ by Corollary~\ref{cor:abeliandiscretecase}.
	\qed{}
\end{proof}

\begin{corollary}\label{cor:abelianhaarzeroone}
	Suppose $G$ is locally compact with Haar measure $\lambda$, and that $G$ is
	connected and Abelian as well.
	If $\mu$ is not spread out, then $\lambda(G_s) = 0$.
\end{corollary}

\begin{proof}
	The measurability of $G_s$ is proved in the upcoming
	Corollary~\ref{cor:couplingmeasurability}.
	Here it is taken for granted.
	If $\lambda(G_s)>0$, then $G_s = G_s-G_s$ contains a neighborhood of the
    identity by the Steinhaus Theorem~\cite{stromberg1972elementary}.
	In this case it follows as in the proof of
	Corollary~\ref{cor:connectedspreadoutcase} that $G_s=G$, which implies
	that $\mu$ is spread out by the same corollary.
	\qed{}
\end{proof}

\section{Properties of The Successful Exact Coupling Set}\label{sec:structureofcouplingset}

The primary goal of this section is to treat the measurability issues previously neglected.
In the Abelian case, the successful exact coupling set is Borel measurable.
To show this, a slight but natural extension of Exercise 6.10.72\ in~\cite{bogachev2007measure2},
is required.
The following gives the existence of a measurable choice of a family of Radon-Nikodym derivatives.
Importantly, the following does not assume absolute continuity and instead produces
Radon-Nikodym derivatives of the absolutely continuous parts of measures.

\begin{proposition}\label{prop:acpart}
    Let $(X,\A,\mu)$ be a finite measure space with $\A$ countably generated, and 
    let $(T,\B)$ be a measurable space.
    Let $\{\mu_t\}_{t \in T}$ be any family of finite measures on $X$ such that
    for each $A \in \A$,  the function $t \mapsto \mu_t(A)$ is $\B$-measurable.
    Then there is an $\A \otimes \B$-measurable $f:X\times T \to \R$ such that
    for every $t \in T$, $x\mapsto f(x,t)$ is a version of the Radon-Nikodym derivative of
    the absolutely continuous part of $\mu_t$ with respect to $\mu$.
\end{proposition}

\begin{proof}
    First consider $X:=[0,1]$ and $\A := \B([0,1])$, the Borel sets on $[0,1]$.
    Fix a sequence $\{\epsilon_n\}_{n=0}^\infty$ with $\epsilon_n \searrow 0$.
    For every $t\in T$,
    \begin{equation}\label{eq:generalradonnikodym}
        \lim_{n} \frac{\mu_t(B(x,\epsilon_n))}{\mu(B(x,\epsilon_n))} = \frac{d\mu_{t,a}}{d\mu}(x),\qquad \mu\text{-a.e.\ $x$},
    \end{equation}
    where $\mu_{t,a}$ denotes the absolutely continuous part of $\mu_t$ with respect to $\mu$.
    This follows from, e.g., Theorem 5.8.8.\ in~\cite{bogachev2007measure1}.
    Define
    \begin{equation}
        f(x,t) := \limsup_n \frac{\mu_t(B(x,\epsilon_n))}{\mu(B(x,\epsilon_n))}
    \end{equation}
    for $x \in \supp{\mu}$ and $t \in T$, and $f(x,t) := 0$ otherwise.
    By \eqref{eq:generalradonnikodym}, it suffices to show $f$ is $\A \otimes \B$-measurable.
    Indeed, consider a fixed $n$ and consider the numerator 
    \begin{equation*}
        (x,t) \mapsto \mu_t(B(x,\epsilon_n))= \int_{[0,1]} 1_{|y-x|<\epsilon_n}\,\mu_t(dy).
    \end{equation*}
    Let $g(x,y):= 1_{|y-x|<\epsilon_n}$ and choose a sequence of measurable simple functions
    $\{s_k\}_{k=0}^\infty$ of the form 
    \begin{equation}
        s_k(x,y) := \sum_{i=0}^{m_k} \alpha_{i,k} 1_{x \in A_{i,k}} 1_{y \in B_{i,k}},
    \end{equation}
    with $0 \leq s_k \leq 1$ and $A_{i,k},B_{i,k} \in \B([0,1])$ for each $k$, and $s_k \to g$ as $k\to\infty$.
    Then
    \begin{equation*}
        \int_{[0,1]} 1_{|y-x|<\epsilon_n}\,\mu_t(dy) 
        = \lim_k \sum_{i=0}^{m_k} \alpha_{i,k} 1_{x \in A_{i,k}} \mu_t(B_{i,k}),
    \end{equation*}
    which shows $(x,t) \mapsto \mu_t(B(x,\epsilon_n))$ is a limit of $\A\otimes \B$-measurable
    functions, showing its measurability.
    The argument for the denominator $(x,t)\mapsto \mu(B(x,\epsilon_n))$ is similar and easier.
    It follows that $f$ is $\A\otimes \B$-measurable.

    Next, consider a general $X$ and $\A$.
    Since $\A$ is countably generated, choose an $\A$-measurable $\phi:X \to [0,1]$
    such that $\A = \{\phi^{-1}(B): B \in \B([0,1]) \}$, cf.\ Theorem 6.5.5 in \cite{bogachev2007measure2}.
    Also set 
    \begin{equation}
        \nu := \mu(\phi \in \cdot),\qquad \nu_t := \mu_t(\phi \in \cdot),
    \end{equation}
    for each $t \in T$.
    For each $B \in \B([0,1])$, it holds that $A := \phi^{-1}(B) \in \A$ and $t \mapsto \nu_t(B) = \mu_t(A)$
    is $\B$-measurable.
    By the case where $X=[0,1]$ and $\A=\B([0,1])$, choose $f:[0,1]\times T\to \R$
    that is $\B([0,1])\otimes \B$-measurable and such that for all $t \in T$,
    $f(\cdot,t)$ is a version of the Radon-Nikodym derivative of the absolutely continuous
    part of $\nu_t$ with respect to $\nu$.
    Define $f_0:X\times T\to \R$ by $f_0(x,t) := f(\phi(x),t)$.
    Then $f_0$ is $\A\otimes \B$-measurable.
    Fix $t \in T$ and let $A \in \A$ be given.
    Choose $B \in \B([0,1])$ with $A = \phi^{-1}(B)$.
    Then
    \begin{align*}
        \int_X 1_{x \in A} f_0(x,t)\,\mu(dx) 
        &= \int_X 1_{\phi(x) \in B} f(\phi(x),t)\,\mu(dx)\\
        &= \int_{[0,1]} 1_{y \in B} f(y,t)\,\nu(dy)\\
        &= \nu_{t,a}(B)\\
        &= \mu_{t,a}(A).
    \end{align*}
    Some care should be taken in the last equality, where it is used that 
    the absolutely continuous part of $\mu_t(\phi \in \cdot)$ with respect to $\mu(\phi \in \cdot)$
    is the same as the push-forward with respect to $\phi$ of the absolutely continuous part
    of $\mu_t$ with respect to $\mu$.
    Write 
    \begin{equation*}
        \nu_t = \nu_{t,a} + \nu_{t,s},\qquad \mu_t = \mu_{t,a} + \mu_{t,s},
    \end{equation*}
    with $\nu_{t,a} \ll \nu$ and $\nu_{t,s} \perp \nu$, and $\mu_{t,a} \ll \mu$ and $\mu_{t,s} \perp \mu$.
    Then also 
    \begin{equation*}
        \nu_t = \mu_t(\phi \in \cdot) = \mu_{t,a}(\phi \in \cdot) + \mu_{t,s}(\phi \in \cdot),
    \end{equation*}
    so it suffices to show by the uniqueness of Lebesgue decompositions
    that 
    \begin{equation*}
        \mu_{t,a}(\phi \in \cdot) \ll \nu \text{ and } \mu_{t,s}(\phi \in \cdot) \perp \nu.
    \end{equation*}
    Indeed, if $B\in\B([0,1])$ is such that $0=\nu(B)=\mu(\phi \in B)$, 
    then  $\mu_{t,a}(\phi \in \cdot) = 0$
    because $\mu_{t,a} \ll \mu$.
    Thus $\mu_{t,a}(\phi \in \cdot) \ll \nu$.
    Similarly, choose $A \in \A$ such that $\mu_{t,s}(A^c) = \mu(A) = 0$.
    Choose $B\in \B([0,1])$ with $A = \phi^{-1}(B)$, then compute
    $\mu_{t,s}(\phi \in B^c) = \mu_{t,s}(A^c) = 0$ and $\nu(B) = \mu(\phi \in B)
    = \mu(A) = 0$, so that $\mu_{t,s}(\phi \in \cdot) \perp \nu$.
    The previous use of $\nu_{t,a}(B) = \mu_{t,a}(A)$ is now justified,
    showing that $f_0(\cdot, t)$ is a version of
    the Radon-Nikodym derivative of the absolutely continuous part of $\mu_t$ with respect to $\mu$,
    completing the claim.
    \qed{}
\end{proof}

The following, together with Corollary~\ref{cor:couplinggroup}, gives Theorem~\ref{thm:abelianproblemresolutionteaser} part (e).

\begin{corollary}\label{cor:couplingmeasurability}
    For a probability measure $\nu$ on $G$,
    the maps $x \mapsto \TV{\nu - \theta_x^{-1}\nu}$, $x\mapsto \TV{\nu \wedge \theta_x^{-1}\nu}$, and
    the set $\{x : \nu \wedge \theta_x^{-1}\nu\neq \zero\}$
    are Borel measurable. 
	In particular, if $G$ is Abelian then $G_s=\cup_{n=1}^\infty \{x \in G:
	\mu^n \wedge \theta_x^{-1}\mu^n \neq \zero\}$
    is Borel measurable.
\end{corollary}

\begin{proof}
    Apply Proposition~\ref{prop:acpart} with $X:=T:=G$ and the family of measures
    $\nu_t := \theta_t^{-1}\nu$ for $t \in G$.
    For $A \subseteq G$ open and $t_n \to t \in G$, Fatou's lemma implies that
    \begin{align*}
        \nu_t(A) &= \int_G 1_{x \in tA}\,\nu(dx)\\
        &= \int_G 1_{t^{-1} \in Ax^{-1}}\,\nu(dx)\\
        &\leq \int_G \liminf_n 1_{t_n^{-1} \in Ax^{-1}}\,\nu(dx) \\
        &\leq \liminf_n \int_G 1_{t_n^{-1} \in Ax^{-1}}\,\nu(dx)\\
        &= \liminf_n \nu_{t_n}(A),
    \end{align*}
    so that $t \mapsto \nu_t(A)$ is semicontinuous and hence measurable.
    A monotone class argument shows that $t \mapsto \nu_t(A)$ is measurable for all Borel $A\subseteq G$.
    Thus, Proposition~\ref{prop:acpart} gives a measurable $f:G\times G \to \R$ such that for every $t \in G$,
    $x\mapsto f(x,t)$ is a version of the Radon-Nikodym derivative of the absolutely continuous part
    of $\theta_t^{-1}\nu$
    with respect to $\nu$.
    It follows that 
    \begin{equation}
        M(t):= \int_G \min\{f(x,t),1\}\,\nu(dx)= \TV{\nu \wedge \theta_t^{-1}\nu}
    \end{equation}
    is measurable in $t$.
    Hence
    \begin{equation}
        \TV{\nu - \theta_t^{-1}\nu} = 2 - 2\TV{\nu \wedge \theta_t^{-1}\nu}
    \end{equation} 
    is measurable in $t$, and
    \begin{equation}
        \{t : \nu \wedge \theta_t^{-1}\nu \neq \zero\} = \{t : M(t)>0\}
    \end{equation}
    is measurable as well.
    \qed{}
\end{proof}

It is not known to the author in the non-Abelian case whether $G_s$ is measurable.
Even in the Abelian case though, little is known
about other structural properties of $G_s$.
When is $G_s$ nicer than Borel
measurable?
The worst case seen so far in Corollary~\ref{cor:Gsanycountableset} is that $G_s$ may be any countable subgroup of $G$, which gives cases where
$G_s$ is an $F_\sigma$ set but not closed (e.g.\ $\Q \subseteq \R$).
Depending on $G$, this also gives cases where $G_s$ is dense (e.g.\ $\Q \subseteq \R$),
infinite but not dense (e.g.\ $\Z\subseteq \R$), and
finite but not trivial (e.g.\ $\{-1,1\} \subseteq \R\setminus\{0\}$).
Corollary~\ref{cor:abelianhaarzeroone} indicates that in many cases
either $G_s=G$ or $\lambda(G_s)=0$, so in these cases the Haar measure on $G$ is not useful to measure the size of $G_s$.
Is there a natural measure with which to measure the size of $G_s$?
What is the Hausdorff dimension of $G_s$, and can it be related, say, to the
Hausdorff dimension of the subgroup generated by $\supp \mu$?
All of these questions remain open and are not investigated further here.

\section{Possible Exact Coupling}\label{sec:possibleexactcoupling}

In this section, a weaker notion of exact coupling is studied.
Suppose that $(S, S^x, T)$ is an exact coupling of $\RW(e,\mu)$ and
$\RW(x,\mu)$.
If $\P(T < \infty) > 0$, then $(S, S^x, T)$ is called a \textdefn{possible
exact coupling}.
The difference between possible exact coupling and successful exact coupling is that a possible exact coupling only requires
$T<\infty$ with positive probability, whereas a successful exact coupling would require $T<\infty$ a.s.

\begin{definition}
	Define the \textdefn{possible exact coupling set} $G_p$ to be the subset of
	all $x \in G$ such that there exists a possible exact coupling of
	$\RW(e,\mu)$ and $\RW(x,\mu)$.
\end{definition}

Carefully looking over the proofs in Section~\ref{sec:mainthm} reveals that in many places,
the fact that a coupling time $T$ satisfies $T<\infty$ a.s.\ is used only 
to guarantee that $\P(T = n)>0$ for some $n$, allowing 
the same proofs work for possible exact couplings as well.
In particular, the following variations on Proposition~\ref{prop:existence} and Theorem~\ref{thm:mainthm}
hold without the need for any kind of assumption about the existence of large sets that commute with $x$.

\begin{proposition}\label{prop:possibleexactcouplingexistence}
	Fix $x \in G$ and suppose that $n \geq 1$ is such that
    $\mu^{n} \geq \nu + \theta_x^{-1}\nu$ for a nonzero measure $\nu$.
	Then $x \in G_p$ and there exists a possible exact coupling of $\RW(e,\mu)$ and
	$\RW(x,\mu)$ with a coupling time
    $T$ satisfying $\P(T = n) = \nu(G)$.
\end{proposition}

\begin{proof}
    In the proof of Proposition~\ref{prop:existence}, the only place where the assumption that
	there exists a $B$ with $\mu^{n}(B)=1$ such that $x$ commutes with all of $B$
    is needed is to show that the constructed coupling time $T$ is a.s.\ finite and $T/n$
    looks like a hitting time of a random walk.
    When this assumption is not met, the coupling from that proof still works, and the coupling time $T$ 
    still satisfies $\P(T=n)=\nu(G)$, but not necessarily $\P(T<\infty)=1$, and $T/n$ does not necessarily look
    like a hitting time of a random walk on $\langle x\rangle$.
    \qed{}
\end{proof}

\begin{theorem}\label{thm:possibleexactcouplingmainthm}
	For all $x \in G$, there exists a possible exact coupling of $\RW(e,\mu)$
	and $\RW(x,\mu)$ with a coupling time $T$ satisfying $\P(T=n)>0$
	if and only if
    $\mu^{n}\wedge \theta_x^{-1}\mu^{n} \neq \zero$.
	In particular, $G_p =\{ x \in G: \exists n \geq 1, \mu^{n}\wedge
	\theta_x^{-1}\mu^{n} \neq \zero\}$.
\end{theorem}

\begin{proof}
    The proof is nearly identical to that of Theorem~\ref{thm:mainthm}, 
    except one appeals to Proposition~\ref{prop:possibleexactcouplingexistence} to construct a possible exact coupling instead of Proposition~\ref{prop:existence}.
    \qed{}
\end{proof}

One may now reap some low-hanging fruit. In particular, it is shown that the possible exact coupling set is Borel measurable,
that in the Abelian case admitting a possible exact coupling and admitting a successful exact coupling are the same, and that if an $n$-fold convolution of a measure
overlaps with one of its shifts, then all higher-fold convolutions of the measure admit the same property.

\begin{corollary}\label{cor:possibleexactcouplingsetmeasurable}
	$G_p$ is Borel measurable.
\end{corollary}

\begin{proof}
    The set in question, by Theorem~\ref{thm:possibleexactcouplingmainthm}, equals
    $\bigcup_{n=1}^\infty \{y: \mu^{n} \wedge \theta_y^{-1}\mu^{n} \neq \zero\}$,
    which is Borel measurable by Corollary~\ref{cor:couplingmeasurability}.
    \qed{}
\end{proof}

\begin{corollary}
    Suppose $G$ is Abelian.
    Then $G_p = G_s$.
\end{corollary}

\begin{proof}
    By Theorems~\ref{thm:possibleexactcouplingmainthm}
	and~\ref{cor:abeliancase}, both equal
	$\{ x \in G: \exists n \geq 1, \mu^{n}\wedge
	\theta_x^{-1}\mu^{n} \neq \zero\}$.
    \qed{}
\end{proof}

Note that the previous corollary says that if an exact coupling with coupling
time $T$ satisfying
$\P(T<\infty)>0$ exists, then an exact coupling with coupling time $T'$ with
$\P(T'<\infty)=1$ exists.
It does not show that if $\P(T<\infty)>0$ then $\P(T<\infty)=1$.

\begin{corollary}
    For a probability measure $\nu$ on $G$, if $\nu^{n_0} \wedge \theta_x^{-1}\nu^{n_0} \neq \zero$
    for some $n_0 \geq 1$, then $\nu^{n} \wedge \theta_x^{-1}\nu^{n} \neq \zero$ for all $n \geq n_0$.
\end{corollary}

\begin{proof}
    Let $n_0$ as above and let $n\geq n_0$ be given. 
	By Theorem~\ref{thm:possibleexactcouplingmainthm},
	choose a possible exact coupling $(S, S^x, T)$ of $\RW(e,\nu)$ and
	$\RW(x,\nu)$ with $\P(T=n_0)>0$.
    Then $T'  := T+(n-n_0)$ is also a coupling time for $S$ and  $S^x$ with $\P(T' = n)>0$,
    so by Theorem~\ref{thm:possibleexactcouplingmainthm} it holds that $\nu^{n}\wedge \theta_x^{-1}\nu^{n} \neq \zero$.
    \qed{}
\end{proof}

In the Abelian case, admitting a possible exact coupling and admitting a successful
exact coupling turned out to be the same.
Lastly, it is shown that in the non-Abelian case this is not necessarily the case.

\begin{example}
    Let $G:=\mathbb{F}_2$ be the free group on two letters $a,b$ and consider $S$ and $S^{ab}$ simple random walks on $G$.
    That is, the step-length distribution $\mu$ is supported on four atoms:
    \begin{equation}
        \mu(\{a\})=\mu(\{a^{-1}\}) = \mu(\{b\}) = \mu(\{b^{-1}\}) = \frac14.
    \end{equation}
    Suppose $S$ starts at the empty word $e$, and $S^{ab}$ starts at $ab$.
    If $S$ and $S^{ab}$ are taken to be independent, then with positive probability
    $S_1=a = S^{ab}_1$, so a possible exact coupling can be easily constructed.
    Furthermore, note that the length $\len S$ of $S$ is itself a Markov chain on $\N$.
    In fact, with $W$ denoting a simple random walk on $\Z$ having
    probability $1/4$ of decreasing and $3/4$ of increasing at each step, and which,
    for any $x \in \Z$, is started at $x$ under a measure $\P_x$, one has
    \begin{equation}
        \P(\len S \text{ returns to } 0) = \P_1(W \text{ hits } 0) <1,
    \end{equation}
    where the last inequality is a standard fact about asymmetric simple random walks on $\Z$.
    It follows that $0$ is a transient state for the Markov chain $\len S$ and, since the chain
    is irreducible, all states are transient.
    Hence a.s.\ the length of $S$ tends to $\infty$ and a limiting word is finalized.
    A similar statement holds for $S^{ab}$.
    Denote the limiting words $\lim S$ and $\lim S^{ab}$.
	Admitting a successful exact coupling is also equivalent, 
	cf.\ Theorem 9.4 in Section 9.5 of \cite{thorisson2000coupling},
	to
	\begin{align}\label{eq:tailmeasurableequality}
		\P(S \in B) = \P(S^{ab} \in B),\qquad B \in \T,
	\end{align}
	where $\T$ is the $\sigma$-algebra of tail measurable events.
    The set 
    \begin{equation*} 
        \{s=\{s_n\}_{n=0}^\infty : \lim s \text{ starts with $b$}\}
    \end{equation*}
    is tail measurable.
    With $\tau$ the hitting time of $e$ for $S^{ab}$, 
    by the strong Markov property and the fact that at time $\tau$ it holds that $S^{ab}$ starts anew as a copy of $S$,
    \begin{align*}
        \P(\lim S^{ab} \text{ starts with $b$})
        &= \P(\tau<\infty)\P(\lim S \text{ starts with $b$})\\
        &= \P_2(W \text{ hits } 0)\P(\lim S \text{ starts with $b$})\\
        &< \P(\lim S \text{ starts with $b$}).
    \end{align*}
    Thus there is no successful exact coupling between $S$ and $S^{ab}$.
\end{example}

\bibliographystyle{spmpsci}
\nocite{*}
\bibliography{exact-couplings.bbl}

\begin{thebibliography}{10}
\providecommand{\url}[1]{{#1}}
\providecommand{\urlprefix}{URL }
\expandafter\ifx\csname urlstyle\endcsname\relax
  \providecommand{\doi}[1]{DOI~\discretionary{}{}{}#1}\else
  \providecommand{\doi}{DOI~\discretionary{}{}{}\begingroup
  \urlstyle{rm}\Url}\fi

\bibitem{arnaldsson2010coupling}
Arnaldsson, {\"O}.: On coupling of discrete random walks on the line (2010)

\bibitem{berbee1979random}
Berbee, H.C.: Random walks with stationary increments and renewal theory.
\newblock MC Tracts \textbf{112}, 1--223 (1979)

\bibitem{bogachev2007measure2}
Bogachev, V.I.: Measure theory, vol.~2.
\newblock Springer Science \& Business Media (2007)

\bibitem{bogachev2007measure1}
Bogachev, V.I.: Measure theory, vol.~1.
\newblock Springer Science \& Business Media (2007)

\bibitem{herrmann1965glattungseigenschaften}
Herrmann, H.: Gl{\"a}ttungseigenschaften der Faltung (1965)

\bibitem{levin2017markov}
Levin, D.A., Peres, Y.: Markov chains and mixing times, vol. 107.
\newblock American Mathematical Soc. (2017)

\bibitem{stam1966shifting}
Stam, A.: On shifting iterated convolutions i.
\newblock Compositio Math \textbf{17}, 268--280 (1966)

\bibitem{stam1967shifting}
Stam, A.: On shifting iterated convolutions ii.
\newblock Compositio Mathematica \textbf{18}(3), 201--228 (1967)

\bibitem{stromberg1972elementary}
Stromberg, K.: An elementary proof of steinhaus’s theorem.
\newblock Proceedings of the American Mathematical Society \textbf{36}(1), 308
  (1972)

\bibitem{thorisson2000coupling}
Thorisson, H.: Coupling, stationarity, and regeneration, vol. 200.
\newblock Springer New York (2000)

\bibitem{thorisson2011open}
Thorisson, H.: Open problems in renewal, coupling and palm theory.
\newblock Queueing Systems \textbf{68}(3-4), 313--319 (2011)

\end{thebibliography}

\end{document}